\documentclass[12pt,reqno]{amsart}

\usepackage{amssymb}

\numberwithin{equation}{section}

\newtheorem{theorem}{Theorem}[section]
\newtheorem{proposition}[theorem]{Proposition}
\newtheorem{lemma}[theorem]{Lemma}
\newtheorem{corollary}[theorem]{Corollary}

\theoremstyle{definition}
\newtheorem{definition}[theorem]{Definition}
\newtheorem{remark}[theorem]{Remark}

\begin{document}

\baselineskip=15.5pt

\title[Connections and genuinely ramified maps]{Connections and genuinely ramified maps
of curves}

\author[I. Biswas]{Indranil Biswas}

\address{School of Mathematics, Tata Institute of Fundamental
Research, Homi Bhabha Road, Mumbai 400005, India}

\email{indranil@math.tifr.res.in}

\author[F.-X. Machu]{Francois-Xavier Machu}

\address{ESIEA, 74 bis Av. Maurice Thorez, 94200 Ivry-sur-Seine, France}

\email{fx.machu@gmail.com}

\author[A. J. Parameswaran]{A. J. Parameswaran}

\address{School of Mathematics, Tata Institute of Fundamental
Research, Homi Bhabha Road, Mumbai 400005, India}

\email{param@math.tifr.res.in}



\date{}

\begin{abstract}
Given a singular connection $D$ on a vector bundle $E$ over an irreducible smooth projective
curve $X$, defined over an algebraically closed field, we show that there is a unique maximal
subsheaf of $E$ on which $D$ induces a nonsingular connection. Given a generically smooth
map $\phi\, :\, Y\, \longrightarrow\, X$ between irreducible smooth projective
curves, and a singular connection $(V,\, D)$ on $Y$, the direct image $\phi_*V$ has
a singular connection. Let $\textbf{R}(\phi_*{\mathcal O}_Y)$ be the unique maximal subsheaf
on which the singular connection on $\phi_*{\mathcal O}_Y$ --- corresponding
to the trivial connection on ${\mathcal O}_Y$ --- induces a nonsingular connection.
We prove that the homomorphism of \'etale fundamental groups $\phi_*\,:\,
\pi_1^{\rm et}(Y,\, y_0) \, \longrightarrow\, \pi_1^{\rm et}(X,\, \phi(y_0))$ induced by
$\phi$ is surjective if and only if 
${\mathcal O}_X\, \subset\, \textbf{R}(\phi_*{\mathcal O}_Y)$ is the unique maximal
semistable subsheaf.

When the characteristic of the base field is zero, this homomorphism $\phi_*$ is surjective if and only if ${\mathcal O}_X
\, =\, \textbf{R}(\phi_*{\mathcal O}_Y)$. For any nonsingular connection $D$ on a vector bundle $V$ over $X$, there
is a natural map $V\,\hookrightarrow\, {\bf R}(\phi_*\phi^*V)$. When the characteristic of the base field is
zero, we prove that the map $\phi$ is genuinely ramified if and only if $V\, =\, {\bf R}(\phi_*\phi^*V)$.
\end{abstract}

\maketitle

\section{Introduction}

Let $X$ and $Y$ be irreducible smooth projective
curves, defined over an algebraically closed field $k$, and let
$\phi\, :\, Y\, \longrightarrow\, X$ be a morphism which is generically smooth (in other
words, $\phi$ is surjective and separable). Fix a base point $y_0\, \in\, Y$, and
consider the homomorphism of \'etale fundamental groups $\phi_*\,:\,
\pi_1^{\rm et}(Y,\, y_0) \, \longrightarrow\, \pi_1^{\rm et}(X,\, \phi(y_0))$ induced by
$\phi$. The map $\phi$ is called genuinely ramified if $\phi_*$
is surjective \cite{BP}. There are many equivalent formulations of
the property of being genuinely ramified, which we recall below.

A map $\phi$ as above is genuinely ramified if and only if one (hence all)
of the following equivalent conditions holds (see \cite{BP}):
\begin{enumerate}
\item The map $\phi$ does not factor through some nontrivial \'etale cover of
$X$ (in particular, $\phi$ is not nontrivial \'etale).

\item The fiber product $Y\times_X Y$ is connected.

\item $\dim H^0(Y,\, \phi^*\phi_*{\mathcal O}_Y)\,=\,1$.

\item The maximal semistable subbundle of the direct image
$\phi_*{\mathcal O}_Y$ is ${\mathcal O}_X$.

\item For every stable vector bundle $E$ on $X$, the pulled back vector bundle
$\phi^*E$ on $Y$ is also stable.
\end{enumerate}
The main theorem of \cite{BP} says that the third statement in the above list holds
for $\phi$ if and only if the fifth statement in the above list holds.

Our aim is to understand the direct image of connections and to interpret the
genuinely ramified maps using the direct image of a particular connection.

Let $E$ be a vector bundle on $X$ equipped with a singular connection $D$. We prove that
there is a unique maximal subsheaf
$$
\textbf{R}(E)\, \subset\, E
$$
on which $D$ induces a nonsingular connection (see Lemma \ref{lem1} and
Definition \ref{def1}).

If $V$ is a vector bundle on $Y$ equipped with a singular connection $D$, then the direct 
image $\phi_*V$ on $X$ is equipped with a natural singular connection (see Lemma \ref{lem2}). Now 
set $(V,\, D)$ to be ${\mathcal O}_Y$ equipped with the trivial connection given by the de 
Rham differential $d$. Let $d_\phi$ denote the singular connection on $\phi_*{\mathcal 
O}_Y$ given by the trivial connection on ${\mathcal O}_Y$. We prove the following (see
Proposition \ref{prop1}):

{\it The map $\phi$ is \'etale if and only if the connection $d_\phi$ is nonsingular.}

Let
$$
\textbf{R}(\phi_*{\mathcal O}_Y)\, \subset\, \phi_*{\mathcal O}_Y
$$
be the unique maximal subsheaf on which $d_\phi$ induces a nonsingular connection.
Then we have
\begin{equation}\label{ie1}
{\mathcal O}_X\, \subset\, \textbf{R}(\phi_*{\mathcal O}_Y)\, \subset\, \phi_*{\mathcal O}_Y\, .
\end{equation}

We prove the following (see Corollary \ref{cor2}):

\textit{The map $\phi$ is genuinely ramified if and only if the subsheaf 
${\mathcal O}_X\, \hookrightarrow\, \textbf{R}(\phi_*{\mathcal O}_Y)$
in \eqref{ie1} is the (unique) maximal semistable subsheaf.}

We next prove the following (see Theorem \ref{thm1}):

{\it Assume that the characteristic of the base field $k$ is zero.
The morphism $\phi$ is genuinely ramified if and only if the inclusion map
$${\mathcal O}_X\, \hookrightarrow\, {\bf R}(\phi_*{\mathcal O}_Y)$$
in \eqref{ie1} is actually an isomorphism.}

We note that Theorem \ref{thm1} actually fails when the characteristic of the
base field $k$ is positive (see Remark \ref{re1}).

Take a vector bundle $V$ on $X$ equipped with a nonsingular connection $D$. Then
$\phi^*V$ has the pulled back nonsingular connection $\phi^*D$. Denote by $\widehat D$ the
singular connection on $\phi_*\phi^*V$ induced by $\phi^*D$. Let
$$
\textbf{R}(\phi_*\phi^*V)\, \subset\, \phi_*\phi^*V
$$
be the unique maximal subsheaf on which $\widehat D$ induces a nonsingular connection.

We prove the following (see Proposition \ref{prop3} and Proposition \ref{prop-l}):

{\it Assume that the characteristic of the base field $k$ is zero.
There is a natural map $V\, \hookrightarrow\, {\bf R}(\phi_*\phi^*V)$.
The map $\phi$ is genuinely ramified if and only if $V\, =\, {\bf R}(\phi_*\phi^*V)$.}

Proposition \ref{prop-l} was kindly pointed out by the referee.

\section{Singular connection on direct image}

The base field $k$ is assumed to be algebraically closed. Let $X$ be an irreducible
smooth projective curve defined over $k$. Fix a finite subset
$$
S\, :=\, \{x_1,\, \cdots , \, x_n\}\, \subset\, X\, .
$$
The reduced effective divisor $\sum_{i=1}^n x_i$ on $X$ will also be denoted
by $S$. The cotangent bundle of $X$ will be denoted by $K_X$.

Let $E$ be a vector bundle on $X$. A connection on $E$ is a differential operator
of order one
$$
D\, :\, E\, \longrightarrow\, E\otimes K_X
$$
such that
\begin{equation}\label{e1}
D(fs) =\, fD(s) +s\otimes df
\end{equation}
for every locally defined function $f$ on $X$ and
every locally defined section $s$ of $E$. A singular connection on $E$ with poles of order $m$
on points of $S$ is a differential operator of order one
$$
D\, :\, E\, \longrightarrow\, E\otimes K_X\otimes {\mathcal O}_X(mS)
$$
such that \eqref{e1} holds. A logarithmic connection on $E$ singular over $S$ is a singular
connection on $E$ with poles of order one on $S$. A singular connection on $E$ with poles
over $S$ is a singular connection on $E$ with poles of order $m$ on $S$ for some $m$.

We now give some examples of connections. When the characteristic of $k$ is zero, a vector 
bundle $E$ on $X$ admits a (nonsingular) connection if and only if every direct summand of 
$E$ (this includes $E$) is of degree zero \cite{We}, \cite{At} (in \cite{At} and \cite{We}
this is proved under the assumption that the base field is complex numbers; see
\cite[p.~ 145, Proposition 3.1]{BS} for the general case). If the characteristic of 
$k$ is $p\, > \, 0$, and $E$ is a vector bundle on $X$ admitting a connection, then the 
degree of every direct summand of $E$ is a multiple of $p$ \cite[p.~ 145, Proposition 3.1]{BS}. If
the characteristic of $k$ is positive, and $F_X\, :\, X\, \longrightarrow\, X$ is the absolute Frobenius 
morphism of $X$, then for any vector bundle $E$ on $X$, the pulled back vector bundle $F^*_X E$ 
has a natural connection (see \cite{Ka}, \cite{Gi}). Moreover, a subsheaf $V\, \subset\, 
F^*_X E$ is the pullback of a subsheaf of $E$ if and only if $V$ is preserved by this 
natural connection on $F^*_X E$.

\begin{lemma}\label{lem1}
Let $E$ be a vector bundle on $X$ and $D$ a singular connection on $E$
with poles of order $m$ on $S$. Then there is a unique maximal subsheaf $F$ of
$E$ on which $D$ induces a (nonsingular) connection.
\end{lemma}

\begin{proof}
Take coherent subsheaves $F_1,\, F_2\, \subset\, E$ such that
$$
D(F_i) \, \subset\, F_i\otimes K_X \, \subset\, E\otimes K_X\otimes {\mathcal O}_X(mS)
$$
for $i\,=\, 1,\, 2$. Then the coherent subsheaf $F_1+ F_2\, \subset\, E$, generated by 
$F_1$ and $F_2$, clearly satisfies the condition that
$$
D(F_1+F_2) \, \subset\, (F_1+F_2)\otimes K_X \, \subset\,
E\otimes K_X\otimes {\mathcal O}_X(mS)\, .
$$
The lemma follows immediately from this observation.
\end{proof}

\begin{remark}\label{rem1}
The maximal subsheaf $F\, \subset\, E$ in Lemma \ref{lem1} on which $D$ induces
a (nonsingular) connection need not be a subbundle, or in other words, $E/F$ need not
be torsionfree. To give an example, consider ${\mathcal O}_X(S)$ equipped with the
logarithmic connection given by the de Rham differential $d$. Then ${\mathcal O}_X\, \subset\,
{\mathcal O}_X(S)$ is the maximal subsheaf on which the logarithmic connection induces
a (nonsingular) connection. The quotient ${\mathcal O}_X(S)/{\mathcal O}_X$ is
a nonzero torsion sheaf.
\end{remark}

\begin{definition}\label{def1}
Let $E$ be a vector bundle on $X$ and $D$ a singular connection on $E$
with poles of order $m$ on $S$. The subsheaf
$$
\textbf{R}(E)\, :=\, F\, \subset\, E
$$
in Lemma \ref{lem1} will be called the \textit{maximal regular subsheaf}.
\end{definition}

\begin{remark}\label{rem1p}
Let $E$ be a vector bundle on $X$ and $D$ a singular connection on $E$
with poles of order $m$ on $S$. It may happen that there is no nonzero subsheaf $F$ of $E$
satisfying the condition that $$D(F) \, \subset\, F\otimes K_X.$$
See the proof of Proposition \ref{prop3} for such an example. If there is no nonzero subsheaf $F$ of $E$
such that $D(F) \, \subset\, F\otimes K_X$, then $\textbf{R}(E)$ in Definition \ref{def1}
is the zero subsheaf of $E$.
\end{remark}

Let $X$ and $Y$ be irreducible smooth projective curves defined over $k$, and let
\begin{equation}\label{e2}
\phi\, :\, Y\, \longrightarrow\, X
\end{equation}
be a generically smooth morphism. Let $S\, \subset\, X$ be the smallest subset such that
$\phi$ is \'etale over the complement $X\setminus S$. The reduced inverse image
$\phi^{-1}(S)_{\rm red}$ will be denoted by $S_Y$.

\begin{lemma}\label{lem2}
Let $E$ be a vector bundle on $Y$, and let $D$ be a singular connection on $E$ with
poles of order $m$ on $S_Y$. Then $D$ produces a singular connection on the direct
image $\phi_*E$ with poles over $S$.
\end{lemma}

\begin{proof}
Denote the complements $X\setminus S$ and $Y\setminus S_Y$ by $X'$ and $Y'$ respectively.
The restriction of $\phi$ to $Y'$ will be denoted by $\widehat{\phi}$. Note
that $\widehat{\phi}^* K_{X'}\,=\, K_{Y'}$. The restriction of
$D$ (respectively, $E$) to $Y'\, \subset\, Y$ will be denoted by $D'$ (respectively,
$E'$). Taking the direct image of the operator
$$
D'\, :\, E'\, \longrightarrow\, E'\otimes K_{Y'}
$$
we get
$$
\widehat{\phi}_*D'\, :\, \widehat{\phi}_*E'\, \longrightarrow\,
\widehat{\phi}_*(E'\otimes K_{Y'})\,=\, \widehat{\phi}_*(E'\otimes\widehat{\phi}^* K_{X'})\,.
$$
The projection formula (see \cite[p.~124, Ex.~5.1(d)]{Ha}) gives that $\widehat{\phi}_*(E'
\otimes\widehat{\phi}^* K_{X'})\,=\,(\widehat{\phi}_*E')\otimes K_{X'}$, and hence we have
$$
\widehat{\phi}_*D'\, :\, \widehat{\phi}_*E'\, \longrightarrow\,
(\widehat{\phi}_*E')\otimes K_{X'}\,.
$$
It is straightforward to check that $\widehat{\phi}_*D'$ is a connection on
$\widehat{\phi}_*E'$. This connection $\widehat{\phi}_*D'$ on $\widehat{\phi}_*E'$
extends to a singular connection on $\phi_*E$ with poles over $S$. Indeed, this
follows immediately from the fact that the differential operator $\widehat{\phi}_*D'$
is algebraic. An upper bound
of the order of pole of this connection at any $x\, \in\, S$ is $d'(r'+1)$, where $d'$ is the
maximum of the orders of poles of $D'$ over the points of $\phi^{-1}(x)$ and $r'$ is the maximum of
the ramification orders of $\phi$ at the points of $\phi^{-1}(x)$.
\end{proof}

\begin{corollary}\label{cor1}
The direct image $\phi_*{\mathcal O}_Y$ has a natural singular connection with poles over $S$.
\end{corollary}

\begin{proof}
Consider the trivial connection $f\, \longmapsto\, df$ on ${\mathcal O}_Y$ given by
the de Rham differential. In view of Lemma \ref{lem2}, this connection produces
a singular connection on $\phi_*{\mathcal O}_Y$ with poles over $S$.
\end{proof}

\begin{proposition}\label{prop1}
The map $\phi$ in \eqref{e2} is \'etale if and only if the (possibly singular) connection
on $\phi_*{\mathcal O}_Y$ obtained in Corollary \ref{cor1} is actually nonsingular.
\end{proposition}

\begin{proof}
Let $d_\phi$ denote the singular connection on $\phi_*{\mathcal O}_Y$
obtained in Corollary \ref{cor1}.

If the map $\phi$ is \'etale, then $S$ in Corollary \ref{cor1} is the zero divisor.
So in that case $d_\phi$ is actually nonsingular.

To prove the converse, assume that $\phi$ is not \'etale. Take a point $y\, \in\, Y$
where the differential of $\phi$ vanishes. Fix a Zariski
open neighborhood $U\, \subsetneq\, X$ of $\phi(y)$. Let $f$ be a function defined on
$\phi^{-1}(U)$ such that $df(y)\, \not=\, 0$. If $\omega$ is a $1$-form on $U$, then 
$$
\phi^*\omega \, \in\, H^0(\phi^{-1}(U),\, K_{\phi^{-1}(U)})
$$
vanishes at $y$, because the differential of $\phi$ vanishes at $y$. Consequently,
$df$ does not lie in the image of the homomorphism
\begin{equation}\label{f1}
H^0(U,\, K_U) \, \longrightarrow\, H^0(\phi^{-1}(U),\, K_{\phi^{-1}(U)}),\ \ \,
\omega\, \longmapsto \phi^*\omega\, .
\end{equation}

Let $$\widehat{f} \, \in\, H^0\left(U,\, \phi_*{\mathcal O}_Y\big\vert_U\right)$$ be
the section corresponding to $f$. Since $df$ does not lie in the image of the homomorphism
in \eqref{f1}, we conclude that
$$
d_\phi(\widehat{f})\, \notin\, H^0\left(U,\, \left(\phi_*{\mathcal O}_Y\right)\big\vert_U
\otimes K_U\right)\, .
$$
Consequently, the connection $d_\phi$ is definitively singular at the point $\phi(y)$.
\end{proof}

\section{Characterization of the genuinely ramified maps}

Take $X$, $Y$ and $\phi$ as in \eqref{e2}.
For the singular connection on $\phi_*{\mathcal O}_Y$ in Corollary \ref{cor1}, let
\begin{equation}\label{e3}
\textbf{R}(\phi_*{\mathcal O}_Y)\, \subset\, \phi_*{\mathcal O}_Y
\end{equation}
be the maximal regular subsheaf (see Definition \ref{def1}).

Since
$$
H^0(Y,\, \text{Hom}({\mathcal O}_Y,\, {\mathcal O}_Y))\,=\,
H^0(Y,\, \text{Hom}(\phi^*{\mathcal O}_X,\, {\mathcal O}_Y))\,=\,
H^0(X,\, \text{Hom}({\mathcal O}_X,\, \phi_*{\mathcal O}_Y))
$$
(see \cite[p.~110]{Ha}), the identity map of ${\mathcal O}_Y$ produces a
nonzero homomorphism
\begin{equation}\label{e4}
\iota\,\,:\,\, {\mathcal O}_X\,\, \hookrightarrow\,\, \phi_*{\mathcal O}_Y\, .
\end{equation}

\begin{remark}\label{rem2}\mbox{}
\begin{enumerate}
\item The coherent subsheaf ${\mathcal O}_X\, \subset\, \phi_*{\mathcal O}_Y$ in
\eqref{e4} is actually a subbundle. Indeed, this follows immediately from the
fact that for any Zariski open neighborhood $U\, \subset\, X$
of any $x\, \in\, X$, and any $f\, \in\, H^0(U,\, {\mathcal O}_U)$ with $f(x)\, \not=\, 0$,
the section $f\circ\phi\, \in\, H^0(\phi^{-1}(U),\, {\mathcal O}_{\phi^{-1}(U)})$ does
not vanish on any point of $\phi^{-1}(x)$.

\item When the characteristic of the base field $k$ is zero, then $\phi_*{\mathcal O}_Y$ actually
splits as ${\mathcal O}_X\oplus T^*_\phi$, where $T_\phi$ is known as the Tschirnhausen bundle
(see \cite{CLV}). We note that $\phi_*{\mathcal O}_Y$ does not split in general when the characteristic
of $k$ is positive. For example, take $k$ to be of characteristic two, and take $\phi$ to be a nontrivial
\'etale covering of degree two. Then $\phi_*{\mathcal O}_Y$ is a nontrivial extension of a degree zero
line bundle by ${\mathcal O}_X$. Indeed, this follows immediately from the fact that
the group ring $k[{\mathbb Z}/2{\mathbb Z}]$ is not completely reducible as a ${\mathbb Z}/2{\mathbb Z}$ module; the
submodule $$\{\lambda\cdot 1+\lambda\cdot g\,\, \mid\,\, \lambda\, \in\, k\}\,\, \subset\,\, k[{\mathbb Z}/2{\mathbb Z}],$$
where ${\mathbb Z}/2{\mathbb Z}\,=\, \{1,\, g\}$, does not have a complement.
\end{enumerate}
\end{remark}

\begin{proposition}\label{prop2}
The subsheaf ${\mathcal O}_X$ in \eqref{e4} is contained in the
subsheaf ${\bf R}(\phi_*{\mathcal O}_Y)$ in \eqref{e3}.
\end{proposition}

\begin{proof}
Let $d_X$ (respectively, $d_Y$) denote the connection on ${\mathcal O}_X$ (respectively,
${\mathcal O}_Y$) given by the de Rham differential on $X$ (respectively, $Y$). The natural
isomorphism $\phi^*{\mathcal O}_X\, \stackrel{\sim}{\longrightarrow}\, {\mathcal O}_Y$
takes the pulled back connection $\phi^* d_X$ on $\phi^*{\mathcal O}_X$ to the connection
$d_Y$ on ${\mathcal O}_Y$. Using this it follows that the following diagram is commutative:
\begin{equation}\label{o1}
\begin{matrix}
{\mathcal O}_X & \stackrel{\iota}{\longrightarrow} & \phi_*{\mathcal O}_Y\\
\,\,\, \Big\downarrow d_X && \,\,\, \Big\downarrow {\widetilde d}_Y\\
K_X & \stackrel{\iota\otimes \xi}{\longrightarrow} &
(\phi_*{\mathcal O}_Y)\otimes K_X\otimes{\mathcal O}_X(\Delta)
\end{matrix}
\end{equation}
where
\begin{itemize}
\item $\iota$ is the homomorphism in \eqref{e4},

\item ${\widetilde d}_Y$ is the (singular) connection on $\phi_*{\mathcal O}_Y$ 
obtained in Corollary \ref{cor1} from the nonsingular connection $d_Y$ on ${\mathcal O}_Y$,

\item $\Delta$ is the polar divisor for ${\widetilde d}_Y$, and

\item $\xi\, :\, K_X\, \longrightarrow\, K_X\otimes{\mathcal O}_X(\Delta)$ is the natural homomorphism.
\end{itemize}
In view of the commutativity of \eqref{o1}, from the definition of the maximal regular
subsheaf $\textbf{R}(\phi_*{\mathcal O}_Y)$ it follows immediately that
${\mathcal O}_X\, \subset\, \textbf{R}(\phi_*{\mathcal O}_Y)$.
\end{proof}

Let
\begin{equation}\label{d-2}
0 \, \subsetneq\, H_1\, \subset\,
\cdots\, \subset\, H_{b-1}\, \subset\, H_b\,=\, \phi_*{\mathcal O}_Y
\end{equation}
be the Harder--Narasimhan filtration of $\phi_*{\mathcal O}_Y$ (see \cite[\S~1.3]{HL}). We note that $H_1$ is
the unique maximal semistable (also called the maximal destabilizing) subbundle of $\phi_*{\mathcal O}_Y$.

Let
\begin{equation}\label{d-3}
0 \, \subsetneq\, G_1\, \subset\,
\cdots\, \subset\, G_{c-1}\, \subset\, G_c\,=\, \textbf{R}(\phi_*{\mathcal O}_Y)
\end{equation}
be the Harder--Narasimhan filtration of $\textbf{R}(\phi_*{\mathcal O}_Y)$. As before,
$G_1$ is the unique maximal semistable subbundle of $\textbf{R}(\phi_*{\mathcal O}_Y)$.

\begin{proposition}\label{prop-1}
The subsheaf $G_1\, \subset\, \textbf{R}(\phi_*{\mathcal O}_Y)\, \subset\,\phi_*{\mathcal O}_Y$
in \eqref{d-3} coincides with the subsheaf $H_1$ in \eqref{d-2}.
\end{proposition}

\begin{proof}
We know that
\begin{equation}\label{d-1}
\text{degree}(H_1)\,=\, 0
\end{equation}
(see \cite[p.~12825, (2.7)]{BP}). So $\text{degree}(G_1)\,\leq\, 0$, because
$\textbf{R}(\phi_*{\mathcal O}_Y)\, \subset\,\phi_*{\mathcal O}_Y$. On the other hand,
$\text{degree}(G_1)\,\geq\, 0$, because ${\mathcal O}_X\, \subset\, \textbf{R}
(\phi_*{\mathcal O}_Y)$ (see Proposition \ref{prop2}). These together imply that
\begin{equation}\label{g-1}
\text{degree}(G_1)\,=\, 0 .
\end{equation}
In view of \eqref{d-1} and \eqref{g-1}, from the properties of the
Harder--Narasimhan filtration we conclude that
\begin{equation}\label{g-2}
G_1\, \subset\, H_1.
\end{equation}

We will now show that
\begin{equation}\label{g-3}
H_1\, \subset\, G_1.
\end{equation}

Recall that ${\mathcal O}_X\, \subset\, H_1\, \subset\, \phi_*{\mathcal O}_Y$. First assume that
$H_1\,=\, {\mathcal O}_X$. But we have ${\mathcal O}_X\, \subset\, G_1$, because ${\mathcal O}_X\, \subset\,
\textbf{R} (\phi_*{\mathcal O}_Y)$ (see Proposition \ref{prop2}) and \eqref{g-1} holds. Therefore, in this
case \eqref{g-3} holds.

Next assume that ${\mathcal O}_X\, \subsetneq\, H_1$. Then there is a nontrivial \'etale covering
$$
g\, :\, \widetilde{X}\, \longrightarrow\, X
$$
and a morphism $h\, :\, Y\, \longrightarrow\,\widetilde{X}$ such that
\begin{enumerate}
\item $g\circ h\,=\, \phi$, and

\item the subsheaf
\begin{equation}\label{g-4}
\iota\,:\, g_*{\mathcal O}_{\widetilde{X}}\, \hookrightarrow \, \phi_*{\mathcal O}_Y,
\end{equation}
given by the factoring of $\phi$ in (1), coincides with $H_1$.
\end{enumerate}
(See the proof of $(2)\, \implies\, (1)$ in the proof of \cite[Proposition 2.6]{BP}.) Since
$g$ is \'etale, from Proposition \ref{prop1} it follows that the connection on
$g_*{\mathcal O}_{\widetilde{X}}$ obtained in Corollary \ref{cor1} is nonsingular.
On the other hand, homomorphism $\iota$ in \eqref{g-4} is clearly connection preserving. In other words, the
following diagram is commutative:
$$
\begin{matrix}
g_*{\mathcal O}_{\widetilde{X}} & \stackrel{\iota}{\longrightarrow} & \phi_*{\mathcal O}_Y\\
\Big\downarrow && \,\,\,\Big\downarrow d_\phi\\
(g_*{\mathcal O}_{\widetilde{X}})\otimes K_X & \stackrel{\iota\otimes \xi}{\longrightarrow} &
(\phi_*{\mathcal O}_Y)\otimes K_X\otimes{\mathcal O}_X(\Delta)
\end{matrix}
$$
where $\Delta$ is the polar divisor for the singular connection ${\widetilde d}_Y$ (see \eqref{o1}) and
$\xi$ is the homomorphism in \eqref{o1}. Therefore, we conclude that
$$
g_*{\mathcal O}_{\widetilde{X}}\, \subset\,\textbf{R}(\phi_*{\mathcal O}_Y).
$$
This again implies \eqref{g-3}.
\end{proof}

Fix a base point $y_0\, \in\, Y$.
The morphism $\phi$ in \eqref{e2} is called a genuinely ramified map if the corresponding
homomorphism of \'etale fundamental groups
$$
\phi_*\,\,:\,\,\pi_1^{\rm et}(Y,\, y_0) \, \longrightarrow\, \pi_1^{\rm et}(X,\, \phi(y_0))
$$
is surjective \cite{BP}.

{}From \eqref{g-1} and Proposition \ref{prop2} it follows that
\begin{equation}\label{g-5}
{\mathcal O}_X\, \hookrightarrow\, G_1.
\end{equation}

\begin{corollary}\label{cor2}
The map $\phi$ is genuinely ramified if and only if the inclusion map
${\mathcal O}_X\, \hookrightarrow\, G_1$ in \eqref{g-5} is an isomorphism.
\end{corollary}

\begin{proof}
{}From \cite[p.~12828, Proposition 2.6]{BP} we know that $\phi$ is genuinely ramified
if and only if $H_1\, =\, {\mathcal O}_X$, where $H_1$ is the subsheaf in \eqref{d-2}.
So the result follows immediately from Proposition \ref{prop-1}.
\end{proof}

\begin{theorem}\label{thm1}
Assume that the characteristic of the base field $k$ is zero.
The morphism $\phi$ in \eqref{e2} is genuinely ramified if and only if 
the inclusion map
$${\mathcal O}_X\, \hookrightarrow\, {\bf R}(\phi_*{\mathcal O}_Y)$$
in Proposition \ref{prop2} is an isomorphism.
\end{theorem}

\begin{proof}
First assume that $\phi$ is not a genuinely ramified map. This implies that
there is a nontrivial \'etale covering
$$
\gamma\,\, :\,\, Z\, \longrightarrow\, X
$$
and a morphism $\beta\, :\, Y\, \longrightarrow\, Z$ such that $\phi\,=\, \gamma\circ\beta$
\cite[Proposition 2.6]{BP}. Since $\phi\,=\, \gamma\circ\beta$, we have
\begin{equation}\label{e5}
\iota\, \, :\,\, \gamma_*{\mathcal O}_Z\, \hookrightarrow\, \phi_*{\mathcal O}_Y\, .
\end{equation}

The possibly singular connection on $\gamma_*{\mathcal O}_Z$ (respectively,
$\phi_*{\mathcal O}_Y$) given by Corollary \ref{cor1} will be denoted by $d_\gamma$
(respectively, $d_\phi$). Since $\gamma$ is \'etale from Proposition \ref{prop1}
we know that $d_\gamma$ is nonsingular.
On the other hand, the homomorphism $\iota$ in \eqref{e5} intertwines $d_\gamma$ and
$d_\phi$, meaning the following diagram is commutative:
$$
\begin{matrix}
\gamma_*{\mathcal O}_Z & \stackrel{\iota}{\longrightarrow} & \phi_*{\mathcal O}_Y\\
\,\,\, \Big\downarrow d_\gamma && \,\,\,\Big\downarrow d_\phi\\
K_X & \stackrel{\iota\otimes \xi}{\longrightarrow} &
(\phi_*{\mathcal O}_Y)\otimes K_X\otimes{\mathcal O}_X(\Delta)
\end{matrix}
$$
where $\Delta$ is the polar divisor for $d_\phi$ and
$\xi\, :\, K_X\, \longrightarrow\,K_X\otimes {\mathcal O}_X(\Delta)$ is the natural homomorphism.
Hence we have
$$
\iota (\gamma_*{\mathcal O}_Z)\, \subset\, \textbf{R}(\phi_*{\mathcal O}_Y)\, .
$$
Since $\text{rank}(\iota (\gamma_*{\mathcal O}_Z))\,=\, \text{degree}(\gamma) \, >\, 1$ (recall that
$\gamma$ is a nontrivial \'etale covering), it
follows that $\text{rank}({\mathcal O}_X)\,< \,\text{rank}(\textbf{R}(\phi_*{\mathcal O}_Y))$.

We will now prove that ${\mathcal O}_X\, =\, \textbf{R}(\phi_*{\mathcal O}_Y)$ if
$\phi$ is a genuinely ramified map.

As before, let $d_\phi$ denote the possibly singular connection on 
$\phi_*{\mathcal O}_Y$ given by Corollary \ref{cor1}. Since $d_\phi$ induces a
nonsingular connection on $\textbf{R}(\phi_*{\mathcal O}_Y)$, and the characteristic of $k$ is zero,
we conclude that
\begin{equation}\label{ed}
\text{degree}(\textbf{R}(\phi_*{\mathcal O}_Y))\,=\, 0
\end{equation}
(see \cite[p.~ 145, Proposition 3.1]{BS}). As in \eqref{d-2}, let
$$
H_1\, \subset\, \phi_*{\mathcal O}_Y
$$
be the maximal semistable subbundle. Since $\mu(H_1)\,=\, 0$ \cite[p.~12825, (2.7)]{BP}, from \eqref{ed} it follows immediately that 
\begin{equation}\label{ed2}
\textbf{R}(\phi_*{\mathcal O}_Y)\, \subset\, H_1\, .
\end{equation}
But $H_1\,=\, {\mathcal O}_X$ if $\phi$ is a genuinely ramified map \cite[p.~12828, Definition 2.5]{BP}. So \eqref{ed2} implies that
$\textbf{R}(\phi_*{\mathcal O}_Y)\, \subset\, {\mathcal O}_X$ is $\phi$ is a genuinely ramified map. Using this and
Proposition \ref{prop2} it follows that ${\mathcal O}_X\, =\, \textbf{R}(\phi_*{\mathcal O}_Y)$ if
$\phi$ is a genuinely ramified map.
\end{proof}

\begin{remark}\label{re1}
Theorem \ref{thm1} is not valid if the assumption that the characteristic of $k$ is zero
is removed. To explain this, assume that the characteristic of $k$ is positive, and let
$$
F_Y\, :\, Y\, \longrightarrow\, Y
$$
be the absolute Frobenius morphism of $Y$. Consider the subsheaf
$$
F^{-1}_Y {\mathcal O}_Y \, \hookrightarrow\, {\mathcal O}_Y;
$$
note that it is not a coherent subsheaf. This inclusion map produces an inclusion map
$$
\phi_* F^{-1}_Y {\mathcal O}_Y \, \hookrightarrow\, \phi_* {\mathcal O}_Y\, .
$$
Let ${\mathcal W}\, \subset\, \phi_* {\mathcal O}_Y$ denote the coherent subsheaf
generated by $\phi_* F^{-1}_Y {\mathcal O}_Y$. Then it is straightforward to check
that the (singular) connection $d_\phi$ on $\phi_* {\mathcal O}_Y$ preserves
$\mathcal W$ and, furthermore, the resulting connection on $\mathcal W$ is nonsingular.
This implies that
$$
{\mathcal W}\, \subset\, {\bf R}(\phi_*{\mathcal O}_Y)\, .
$$
But clearly, $${\mathcal O}_X\, \subsetneq\, {\mathcal W}.$$
To complete the example, if we take $X\,=\, {\mathbb P}^1_k$, then $\phi$ is
genuinely ramified.
\end{remark}

\section{Pullback of irreducible connections}

Take $X$, $Y$ and $\phi$ as in \eqref{e2}.
Let $E$ be a vector bundle on $X$ equipped with a nonsingular connection $D$. The
connection $D$ induces a connection on $\phi^*E$; this induced connection will be
denoted by $\phi^*D$. The singular connection on $\phi_*\phi^*E$, obtained in
Lemma \ref{lem2} from $\phi^*D$, will be denoted by
\begin{equation}\label{ee}
\widehat{D}.
\end{equation}
Let
\begin{equation}\label{f2}
\textbf{R}(\phi_*\phi^*E)\, \hookrightarrow\, \phi_*\phi^*E
\end{equation}
be the maximal regular subsheaf for this singular connection $\widehat{D}$.

Using the projection formula, \cite[p.~124, Ex.~5.1(d)]{Ha}, we have
\begin{equation}\label{e11}
\phi_*\phi^*E\,=\, E\otimes\phi_*{\mathcal O}_Y\, .
\end{equation}
The inclusion map $\iota\, :\, {\mathcal O}_X\, \hookrightarrow\, \phi_*{\mathcal O}_Y$
in \eqref{e4} produces an injective homomorphism
\begin{equation}\label{e10}
H\, :=\, \text{Id}_E\otimes\iota \, \, :\,\, E\,=\, E\otimes {\mathcal O}_X
\, \longrightarrow\, E\otimes\phi_*{\mathcal O}_Y\,=\, \phi_*\phi^*E\,.
\end{equation}
{}From Remark \ref{rem2}(1) we know that $H(E)$ is a subbundle of $\phi_*\phi^*E$.

\begin{proposition}\label{prop3}
Assume that the characteristic of the base field $k$ is zero.
Assume that the map $\phi$ in \eqref{e2} is genuinely ramified. Then
$$
\textbf{R}(\phi_*\phi^*E)\,=\, H(E)
$$
as subsheaves of $\phi_*\phi^*E$ (see \eqref{f2} and \eqref{e10}).
\end{proposition}

\begin{proof}
As before, let $d_\phi$ denote the singular connection on $\phi_*{\mathcal O}_Y$
obtained in Corollary \ref{cor1}. The connection $D$ on $E$ and this singular connection
$d_\phi$ together produce a singular connection on $E\otimes\phi_*{\mathcal O}_Y$;
this singular connection on $E\otimes\phi_*{\mathcal O}_Y$ will be denoted by $\widehat{D}_\phi$.
The identification $E\otimes\phi_*{\mathcal O}_Y\,=\, \phi_*\phi^*E$ in \eqref{e11}
takes $\widehat{D}_\phi$ to $\widehat{D}$ (see \eqref{ee}). Indeed, this follows from the
construction of $\widehat{D}$ (see Lemma \ref{lem2}).
The subsheaf $\iota({\mathcal O}_X)\, \subset\, \phi_*{\mathcal O}_Y$ (see
\eqref{e10}) is preserved by the singular connection $d_\phi$ (see Proposition \ref{prop2}
and its proof). This, and the fact that the identification
$E\otimes\phi_*{\mathcal O}_Y\,=\, \phi_*\phi^*E$ in \eqref{e11}
takes the singular connection on $E\otimes\phi_*{\mathcal O}_Y$ induced
by $D$ and $d_\phi$ to the singular connection $\widehat{D}$ on $\phi_*\phi^*E$, together
imply that the homomorphism $H$ in \eqref{e10} takes the connection $D$ (on $E$)
to $\widehat{D}$ in \eqref{ee}. In other words, the following diagram is commutative:
\begin{equation}\label{o3}
\begin{matrix}
E & \stackrel{H}{\longrightarrow} & \phi_*\phi^*E\\
\,\,\,\, \Big\downarrow D &&\,\,\,\, \Big\downarrow \widehat{D}\\
E\otimes K_X & \stackrel{H\otimes\xi}{\longrightarrow} & E\otimes K_X\otimes {\mathcal O}_X(\Delta)
\end{matrix}
\end{equation}
where $\Delta$ is the polar divisor for $\widehat{D}$ and $\xi\, :\, K_X\, \longrightarrow\,
K_X\otimes {\mathcal O}_X(\Delta)$ is the natural homomorphism; recall that $D$ is a nonsingular connection.
Now the commutativity of \eqref{o3} implies that 
\begin{itemize}
\item the subsheaf
$$
H(E)\, \subset\, \phi_*\phi^*E
$$
is preserved by the singular connection $\widehat{D}$ on $\phi_*\phi^*E$, and

\item $\widehat{D}$ induces a nonsingular connection on $H(E)$.
\end{itemize}
Consequently, we have
\begin{equation}\label{e12}
H(E)\, \subset\, \textbf{R}(\phi_*\phi^*E)\, .
\end{equation}

Note that
$$
(\phi_*\phi^*E)/H(E)\,=\, (E\otimes\phi_*{\mathcal O}_Y)/E\,=\,
E\otimes (\phi_*{\mathcal O}_Y/\iota({\mathcal O}_X))\, .
$$
Since $\phi_*{\mathcal O}_Y/\iota({\mathcal O}_X)$ is locally free (see
Remark \ref{rem2}(1)), the subsheaf
$H(E)\, \subset\, \phi_*\phi^*E$ is a subbundle. Consider the quotient map
\begin{equation}\label{e13}
\textbf{q}\,\,:\,\, \textbf{R}(\phi_*\phi^*E) \, \longrightarrow\,
(\phi_*\phi^*E)/H(E)\,=\, E\otimes (\phi_*{\mathcal O}_Y/\iota({\mathcal O}_X))\, .
\end{equation}
In view of \eqref{e12}, to prove the proposition it suffices to show that $\textbf{q}\,=\, 0$.

Since both $H(E)$ and $\textbf{R}(\phi_*\phi^*E)$ are preserved by the singular
connection $\widehat{D}$ on $\phi_*\phi^*E$,
\begin{enumerate}
\item $\widehat{D}$ induces a singular connection on $(\phi_*\phi^*E)/H(E)$ (recall that
$(\phi_*\phi^*E)/H(E)$ is locally free), and

\item the homomorphism $\textbf{q}$ in \eqref{e13} takes the nonsingular connection
on $\textbf{R}(\phi_*\phi^*E)$ (given by $\widehat{D}$) to the singular connection
on $$(\phi_*\phi^*E)/H(E)\,=\, E\otimes ((\phi_*{\mathcal O}_Y)/\iota({\mathcal O}_X))$$
induced by $\widehat{D}$.
\end{enumerate}

Let
\begin{equation}\label{e14}
\textbf{q}'\,\,:\,\, E^*\otimes \textbf{R}(\phi_*\phi^*E) \, \longrightarrow\,
(\phi_*{\mathcal O}_Y)/\iota({\mathcal O}_X)
\end{equation}
be the homomorphism obtained by composing $${\rm Id}_{E^*}\otimes \textbf{q}\, 
\,:\, E^*\otimes\textbf{R}(\phi_*\phi^*E) \, \longrightarrow\,
E^*\otimes E\otimes ((\phi_*{\mathcal O}_Y)/\iota({\mathcal O}_X))$$ (see \eqref{e13})
with the homomorphism $E^*\otimes E\otimes ((\phi_*{\mathcal O}_Y)/\iota({\mathcal O}_X))
\, \longrightarrow\, (\phi_*{\mathcal O}_Y)/\iota({\mathcal O}_X)$
constructed using the natural pairing $E^*\otimes E\, \longrightarrow\, {\mathcal O}_X$.

Let $D^*$ denote the nonsingular connection on $E^*$ given by the nonsingular connection 
$D$ on $E$. The connection $D^*$, and the connection on $\textbf{R}(\phi_*\phi^*E)$
given by $\widehat{D}$, together produce a connection on $E^*\otimes
\textbf{R}(\phi_*\phi^*E)$; this connection on $E^*\otimes
\textbf{R}(\phi_*\phi^*E)$ will be denoted by ${\mathcal D}_1$.
As noted before, The subbundle $\iota({\mathcal O}_X)\, \subset\, \phi_*{\mathcal O}_Y$
(see \eqref{e10} and Remark \ref{rem2}(1)) is preserved by the singular connection $d_\phi$.
Consequently, $d_\phi$ induces a singular connection on the quotient bundle
$(\phi_*{\mathcal O}_Y)/\iota({\mathcal O}_X)$; this singular connection on
$(\phi_*{\mathcal O}_Y)/\iota({\mathcal O}_X)$ will be denoted by ${\mathcal D}_2$. Since 
$\textbf{q}$ in \eqref{e13} takes the nonsingular connection
on $\textbf{R}(\phi_*\phi^*E)$ (given by $\widehat{D}$) to the singular connection
on $(\phi_*\phi^*E)/H(E)\,=\, E\otimes ((\phi_*{\mathcal O}_Y)/\iota({\mathcal O}_X))$, it
follows immediately that $\textbf{q}'$ in \eqref{e14} takes the above defined
singular connection ${\mathcal D}_1$ on $E^*\otimes
\textbf{R}(\phi_*\phi^*E)$ to the singular connection ${\mathcal D}_2$ on
$(\phi_*{\mathcal O}_Y)/\iota({\mathcal O}_X)$.

Now note that ${\mathcal D}_1$ is a nonsingular connection, because both
$D^*$, and the connection on $\textbf{R}(\phi_*\phi^*E)$ given by $\widehat{D}$, are
nonsingular. On the other hand, since $\phi$ is genuinely ramified, it can be shown
that $(\phi_*{\mathcal O}_Y)/\iota({\mathcal O}_X)$ does not
contain any nonzero subsheaf on which ${\mathcal D}_2$ induces a nonsingular connection.
Indeed, if ${\mathcal D}_2$ induces a nonsingular connection connection on
${\mathcal V}\, \subset\, (\phi_*{\mathcal O}_Y)/\iota({\mathcal O}_X)$, then
consider the inverse image $\widetilde{\mathcal V}\, \subset\, \phi_*{\mathcal O}_Y$
of ${\mathcal V}$ for the quotient map $\phi_*{\mathcal O}_Y\, \longrightarrow\,
(\phi_*{\mathcal O}_Y)/\iota({\mathcal O}_X)$. The singular connection $d_\phi$
on $\phi_*{\mathcal O}_Y$ induces a nonsingular connection on $\widetilde{\mathcal V}$, because
${\mathcal D}_2$ induces a nonsingular connection connection on
${\mathcal V}$. From Theorem
\ref{thm1} it follows that $\widetilde{\mathcal V}\,=\, \iota({\mathcal O}_X)$,
and hence ${\mathcal V}\,=\, 0$. So $(\phi_*{\mathcal O}_Y)/\iota({\mathcal O}_X)$ does not
contain any nonzero subsheaf on which ${\mathcal D}_2$ induces a nonsingular connection.

Since $\textbf{q}'$ in \eqref{e14} takes ${\mathcal D}_1$ to ${\mathcal
D}_2$, and ${\mathcal D}_1$ is a nonsingular connection, while
$(\phi_*{\mathcal O}_Y)/\iota({\mathcal O}_X)$ does not
contain any nonzero subsheaf on which ${\mathcal D}_2$ induces a nonsingular connection,
considering the image of $\textbf{q}'$ it follows that
$$
\textbf{q}'\,=\, 0\, .
$$
This implies that $\textbf{q}$ in \eqref{e13} vanishes identically, and hence
$H(E)\, =\, \textbf{R}(\phi_*\phi^*E)$.
\end{proof}

The following converse of Proposition \ref{prop3} was pointed out by the referee.

\begin{proposition}\label{prop-l}
Assume that the characteristic of the base field $k$ is zero.
Assume that
$$
\textbf{R}(\phi_*\phi^*E)\,=\, H(E)
$$
as subsheaves of $\phi_*\phi^*E$ (see \eqref{f2} and \eqref{e10}).
Then the map $\phi$ in \eqref{e2} is genuinely ramified.
\end{proposition}

\begin{proof}
To prove the proposition by contradiction, assume that the map $\phi$ is not genuinely ramified.
Then, as noted in the proof of Proposition \ref{prop-1}, there is a nontrivial \'etale covering
$$
g\, :\, \widetilde{X}\, \longrightarrow\, X
$$
and a morphism $h\, :\, Y\, \longrightarrow\,\widetilde{X}$ such that
\begin{equation}\label{j1}
g\circ h\,=\, \phi .
\end{equation}
{}From \eqref{j1} it follows immediately that
\begin{equation}\label{j2}
\iota\,\, :\,\, g_*g^*E\,\,\hookrightarrow\,\, \phi_*\phi^*E.
\end{equation}

Consider the nonsingular connection $g^*D$ on $g^*E$ given by $D$. Let
$\widetilde{D}$ denote the connection on $g_*g^*E$ obtained in 
Lemma \ref{lem2} from $g^*D$. We note that this connection 
$\widetilde{D}$ on $g_*g^*E$ is nonsingular. Indeed, this follows from the
observation that $g^*K_X\,=\, K_{\widetilde{X}}$, because the map $g$ is \'etale, and hence
$$g_*((g^*E)\otimes K_{\widetilde{X}})\,=\, (g_*g^*E)\otimes K_X$$
(projection formula).

The map in \eqref{j2} intertwines the connections $\widehat D$ (see \eqref{ee}) and $\widetilde D$ (see above)
on $\phi_*\phi^*E$ and $g_*g^*E$ respectively, meaning the following diagram is commutative:
$$
\begin{matrix}
g_*g^* E & \stackrel{\iota}{\longrightarrow} & \phi_*\phi^*E\\
\,\,\,\, \Big\downarrow {\widetilde D} && \,\,\,\, \Big\downarrow {\widehat D}\\
(g_*g^* E)\otimes K_X & \stackrel{\iota\otimes \xi}{\longrightarrow} &
(\phi_*\phi^*E)\otimes K_X\otimes{\mathcal O}_X(\Delta)
\end{matrix}
$$
where $\Delta$ is the polar divisor for $\widehat D$ and
$\xi\, :\, K_X\, \longrightarrow\,K_X\otimes {\mathcal O}_X(\Delta)$ is the natural homomorphism.
Consequently, we have $$g_*g^* E \, \subset\, \textbf{R}(\phi_*\phi^*E).$$ This implies that
$$\text{rank}(\textbf{R}(\phi_*\phi^*E))
\, \geq \, \text{rank}(g_*g^* E) \,=\, \text{rank}(E)\cdot \text{degree}(g)\, >\,
\text{rank}(E)\,=\, \text{rank}(H(E))$$
because $g$ is a nontrivial \'etale covering. In particular, we have
$$
\textbf{R}(\phi_*\phi^*E)\,\not=\, H(E).
$$
This completes the proof.
\end{proof}

\section*{Acknowledgements}

We are very grateful to the referee for Proposition \ref{prop-l} and also for
helpful comments to improve the manuscript.


\end{document}